\documentclass[12pt]{amsart}
\usepackage{amssymb}
\newtheorem{theorem}{Theorem}
\newtheorem{proposition}[theorem]{Proposition}
\numberwithin{theorem}{section}
\theoremstyle{remark}
\newtheorem*{example}{Example}
\newtheorem*{remark}{Remark}
\newcommand{\on}[2]{\setbox0=\hbox{$#1$}\setbox1=\hbox{$#2$}%
            \dimen0=\wd0\advance\dimen0 by \wd1\divide\dimen0 by 2%
             \ifdim\wd0>\wd1$#1$\hskip-\dimen0$#2$\advance\dimen0 by -\wd1%
              \else$#2$\hskip-\dimen0$#1$\advance\dimen0 by -\wd0%
             \fi%
            \hskip\dimen0}
\newcommand{\cross}{\smash\times\vphantom\bullet}
\newcommand{\onn}[2]{\makebox{\on{#1}{#2}}}
\newcommand{\leftcross}[1]{\hspace*{-1ex}\onn{\raisebox{-.5ex}
     {$\genfrac{}{}{0pt}{}{\hfill\hrulefill}{\hphantom{\hspace*{2ex}#1}}$}}
                {\stackrel{#1}{\cross}}}
\newcommand{\leftblob}[1]{\hspace*{-1ex}\onn{\raisebox{-.5ex}
     {$\genfrac{}{}{0pt}{}{\hfill\hrulefill}{\hphantom{\hspace*{2ex}#1}}$}}
                {\stackrel{#1}{\bullet}}}
\newcommand{\rightblob}[1]{\onn{\raisebox{-.5ex}
     {$\genfrac{}{}{0pt}{}{\hrulefill\hfill}{\hphantom{\hspace*{2ex}#1}}$}}
                 {\stackrel{#1}{\bullet}}\hspace*{-1ex}}
\newcommand{\middleblob}[1]{\onn{\raisebox{-.5ex}
     {$\genfrac{}{}{0pt}{}{\hrulefill}{\hphantom{\hspace*{2ex}#1}}$}}
                 {\stackrel{#1}{\bullet}}}
\newcommand{\middlecross}[1]{\onn{\raisebox{-.5ex}
     {$\genfrac{}{}{0pt}{}{\hrulefill}{\hphantom{\hspace*{2ex}#1}}$}}
                 {\stackrel{#1}{\cross}}}
\newcommand{\ooo}[3]{\leftblob{#1}\hspace*{-1ex}
               \middleblob{#2}\hspace*{-1ex}\rightblob{#3}}
\newcommand{\xoo}[3]{\leftcross{#1}\hspace*{-1ex}
               \middleblob{#2}\hspace*{-1ex}\rightblob{#3}}
\newcommand{\xxo}[3]{\leftcross{#1}\hspace*{-1ex}
               \middlecross{#2}\hspace*{-1ex}\rightblob{#3}}
\begin{document}
\title{The Penrose Transform for Complex Projective Space}
\author{Michael Eastwood}
\address{Department of Mathematics, University of Adelaide,
South Australia 5005}
\email{meastwoo@member.ams.org}
\thanks{This work was supported by the Australian Research Council.}
\begin{abstract}
Various complexes of differential operators are constructed on complex
projective space via the Penrose transform, which also computes their
cohomology. 
\end{abstract}
\renewcommand{\subjclassname}{\textup{2000} Mathematics Subject Classification}
\maketitle

\section{Introduction} Throughout this article ${\mathbb{CP}}_n$ will denote
complex projective space as a homogeneous Riemannian manifold under the natural
action of ${\mathrm{SU}}(n+1)$. The invariant metric is called the Fubini-Study
metric. Details may be found in~\cite{besse}. Let
${\mathbb{F}}_{1,2}({\mathbb{C}}^{n+1})$ denote the complex flag manifold
$$\{(L,P)\mbox{ s.t.\ }0\subset L\subset P\subset{\mathbb{C}}^{n+1},\ 
\dim L=1,\ \dim P=2\}$$
and define a mapping
\begin{equation}\label{twistor}
\tau:{\mathbb{F}}_{1,2}({\mathbb{C}}^{n+1})\to{\mathbb{CP}}_n\quad\mbox{by}
\quad(L,P)\xrightarrow{\;\tau\;}L^\perp\cap P,\end{equation}
where $L^\perp$ is the orthogonal complement of $L\subset{\mathbb{C}}^{n+1}$
with respect to a fixed Hermitian inner product on~${\mathbb{C}}^{n+1}$. Notice
that $\tau$ is a submersion and, although it is not itself holomorphic, its
fibres are clearly holomorphic since over $\ell\in{\mathbb{CP}}_n$ the fibre
consists of pairs $(\ell^\perp\cap P,P)$ for those planes $P$ 
containing~$\ell$. As such, 
$\tau^{-1}(\ell)\cong{\mathbb{P}}({\mathbb{C}}^{n+1}/\ell)$ and is 
intrinsically ${\mathbb{CP}}_{n-1}$ as a complex manifold. 

The classical case is when~$n=2$, for then ${\mathbb{F}}_{1,2}({\mathbb{C}}^3)$
is the twistor space of~${\mathbb{CP}}_2$. It is a special case of the general
construction~\cite{ahs} associating a complex manifold to any anti-self-dual
$4$-dimensional conformal manifold. In fact, with its usual orientation
${\mathbb{CP}}_2$ is a self-dual manifold but we shall adopt the reverse
orientation as in~\cite{mfa,crl}. Equivalently, it is the orientation 
that makes 
the K\"ahler form $J$ on ${\mathbb{CP}}_2$
anti-self-dual. The Penrose
transform for ${\mathbb{CP}}_2$ realises the analytic cohomology
$H^r({\mathbb{F}}_{1,2}({\mathbb{C}}^3),{\mathcal{O}}(V))$ for any holomorphic
homogeneous vector bundle $V$ on ${\mathbb{F}}_{1,2}({\mathbb{C}}^3)$ as the
cohomology of an appropriate elliptic complex on~${\mathbb{CP}}_2$. On the
other hand, for irreducible bundles $V$ the Bott-Borel-Weil Theorem realises
these cohomology spaces for as irreducible representations of
${\mathrm{SL}}(4,{\mathbb{C}})$. So the Penrose transform both constructs
natural elliptic complexes on ${\mathbb{CP}}_2$ and also identifies the
cohomology of these complexes. Here are some examples (see \cite{mfa,es,njh}).

\begin{example} 
That $H^r({\mathbb{F}}_{1,2}({\mathbb{C}}^3),{\mathcal{O}})={\mathbb{C}}$ for
$r=0$ and vanishes for $r\geq 1$ implies
$$0\to{\mathbb{R}}\to\Gamma({\mathbb{CP}}_2,\Lambda^0)\xrightarrow{\;d\;}
\Gamma({\mathbb{CP}}_2,\Lambda^1)\xrightarrow{\;d\;}
\Gamma({\mathbb{CP}}_2,\Lambda_+^2)\to 0$$
is exact. Here the mappings $d$ are induced by the exterior derivative and 
$\Lambda^p$ denotes the bundle of $p$-forms with $\Lambda_+^2$ the self-dual 
$2$-forms. Alternatively, we may use complex notation whence  
\begin{equation}\label{prototype}\raisebox{-15pt}{\begin{picture}(301,40)
\put(0,35){\makebox(0,0){$0$}}
\put(5,35){\vector(1,0){10}}
\put(23,36){\makebox(0,0){${\mathbb{C}}$}}
\put(28,35){\vector(1,0){10}}
\put(70,35){\makebox(0,0){$\Gamma({\mathbb{CP}}_2,\Lambda^{0,0})$}}
\put(103,35){\vector(1,0){24}}
\put(103,32){\vector(1,-1){24}}
\put(113,23){\scriptsize$\partial$}
\put(113,37){\scriptsize$\overline\partial$}
\put(162,35){\makebox(0,0){$\Gamma({\mathbb{CP}}_2,\Lambda^{0,1})$}}
\put(162,5){\makebox(0,0){$\Gamma({\mathbb{CP}}_2,\Lambda^{1,0})$}}
\put(162,20){\makebox(0,0){$\oplus$}}
\put(195,32){\vector(1,-1){24}}
\put(195,5){\vector(1,0){24}}
\put(205,23){\scriptsize$\partial$}
\put(205,7){\scriptsize$\overline\partial$}
\put(253,5){\makebox(0,0){$\Gamma({\mathbb{CP}}_2,\Lambda_\perp^{1,1})$}}
\put(286,5){\vector(1,0){10}}
\put(301,5){\makebox(0,0){$0$}}
\end{picture}}\end{equation}
is exact. Here $\Lambda_\perp^{p,q}$ denotes the bundles of forms of type 
$(p,q)$ and the subscript $\perp$ denotes those that are orthogonal to~$J$. 
With our choice of orientation
$${\mathbb{C}}\otimes_{\mathbb{R}}\Lambda^2=
\begin{array}[t]{ccc}
\underbrace{\Lambda^{1,0}\oplus\Lambda^{0,1}\oplus{\mathbb{C}}J}
&\oplus&\Lambda_\perp^{1,1}\\ \|&&\|\\
{\mathbb{C}}\otimes_{\mathbb{R}}\Lambda_-^2&&
\makebox[0pt]{${\mathbb{C}}\otimes_{\mathbb{R}}\Lambda_+^2$}\end{array}$$
as detailed in~\cite{mfa}.
\end{example}

The aim of this article is to extend the Penrose transform to ${\mathbb{CP}}_n$ 
where (\ref{twistor}) is viewed as the twistor fibration. We shall find, for 
example, an exact sequence
\begin{equation}\label{lopsided}\raisebox{-15pt}{\begin{picture}(391,40)
\put(0,35){\makebox(0,0){$0$}}
\put(5,35){\vector(1,0){10}}
\put(23,36){\makebox(0,0){${\mathbb{C}}$}}
\put(28,35){\vector(1,0){10}}
\put(70,35){\makebox(0,0){$\Gamma({\mathbb{CP}}_3,\Lambda^{0,0})$}}
\put(103,35){\vector(1,0){24}}
\put(103,32){\vector(1,-1){24}}
\put(113,23){\scriptsize$\partial$}
\put(113,37){\scriptsize$\overline\partial$}
\put(162,35){\makebox(0,0){$\Gamma({\mathbb{CP}}_3,\Lambda^{0,1})$}}
\put(162,5){\makebox(0,0){$\Gamma({\mathbb{CP}}_3,\Lambda^{1,0})$}}
\put(162,20){\makebox(0,0){$\oplus$}}
\put(195,35){\vector(1,0){24}}
\put(195,32){\vector(1,-1){24}}
\put(195,5){\vector(1,0){24}}
\put(205,23){\scriptsize$\partial$}
\put(205,37){\scriptsize$\overline\partial$}
\put(205,7){\scriptsize$\overline\partial$}
\put(253,5){\makebox(0,0){$\Gamma({\mathbb{CP}}_3,\Lambda_\perp^{1,1})$}}
\put(253,35){\makebox(0,0){$\Gamma({\mathbb{CP}}_3,\Lambda^{0,2})$}}
\put(253,20){\makebox(0,0){$\oplus$}}
\put(286,32){\vector(1,-1){24}}
\put(286,5){\vector(1,0){24}}
\put(296,23){\scriptsize$\partial$}
\put(296,7){\scriptsize$\overline\partial$}
\put(344,5){\makebox(0,0){$\Gamma({\mathbb{CP}}_3,\Lambda_\perp^{1,2})$}}
\put(376,5){\vector(1,0){10}}
\put(391,5){\makebox(0,0){$0$}}
\end{picture}}\end{equation}
as the counterpart to (\ref{prototype}) on~${\mathbb{CP}}_3$. Notice that the 
bundles occurring in this complex are irreducible on ${\mathbb{CP}}_3$ as an 
Hermitian manifold.

\begin{example} 
The complex generated on ${\mathbb{CP}}_2$ by the holomorphic tangent bundle
$\Theta$ on the twistor space ${\mathbb{F}}_{1,2}({\mathbb{C}}^3)$ is
\begin{equation}\label{asd_deformation}
0\to\Lambda^1\xrightarrow{\;\nabla\;}
\begin{picture}(10,5)
\put(0,0){\line(1,0){10}}
\put(0,5){\line(1,0){10}}
\put(0,0){\line(0,1){5}}
\put(5,0){\line(0,1){5}}
\put(10,0){\line(0,1){5}}
\end{picture}_{\,\circ}\Lambda^1
\xrightarrow{\;\nabla^{(2)}\;}
\begin{picture}(10,10)
\put(0,0){\line(1,0){10}}
\put(0,5){\line(1,0){10}}
\put(0,10){\line(1,0){10}}
\put(0,0){\line(0,1){10}}
\put(5,0){\line(0,1){10}}
\put(10,0){\line(0,1){10}}
\end{picture}_{\,\circ+}\Lambda^1\to 0.\end{equation}
Here, the first differential operator in local c\"oordinates is
\begin{equation}\label{confKilling}
\textstyle\phi_a\mapsto[\nabla_a\phi_b+\nabla_b\phi_a]_\circ=
\nabla_a\phi_b+\nabla_b\phi_a-\frac{2}{n}g_{ab}\nabla^c\phi_c\end{equation}
where the subscript $\circ$ denotes the trace-free part with respect to the
Fubini-Study metric $g_{ab}$ and $\nabla_a$ is the Levi-Civita connection for
this metric. We are using Young diagrams~\cite{fh} to specify the
symmetries of tensor bundles. For example, the bundle
$\begin{picture}(10,5)
\put(0,0){\line(1,0){10}}
\put(0,5){\line(1,0){10}}
\put(0,0){\line(0,1){5}}
\put(5,0){\line(0,1){5}}
\put(10,0){\line(0,1){5}}
\end{picture}_{\,\circ}\Lambda^1$ consists of symmetric trace-free tensors. 
The second operator in (\ref{asd_deformation}) is 
$$\psi_{ab}\mapsto[\nabla_{(a}\nabla_{c)}\psi_{bd}
-\nabla_{(b}\nabla_{c)}\psi_{ad}-\nabla_{(a}\nabla_{d)}\psi_{bc}
+\nabla_{(b}\nabla_{d)}\psi_{ac}]_{\circ+}$$
where $+$ denotes the self-dual part with respect to the skew indices~$ab$. The
bundle $\begin{picture}(10,10) 
\put(0,0){\line(1,0){10}}
\put(0,5){\line(1,0){10}} 
\put(0,10){\line(1,0){10}} 
\put(0,0){\line(0,1){10}}
\put(5,0){\line(0,1){10}} 
\put(10,0){\line(0,1){10}}
\end{picture}_{\,\circ+}\Lambda^1$ is precisely the bundle of covariant tensors
with the resulting symmetries. It is an irreducible ${\mathrm{SO}}(4)$-bundle
of rank~$5$. As observed in~\cite{kokyuroku}, it is straightforward to write
the complexification of (\ref{asd_deformation}) in terms of irreducible
Hermitian bundles:--
\begin{equation}\label{asd_Hermitian}\raisebox{-30pt}{\begin{picture}(275,70)
\put(0,65){\makebox(0,0){$0$}}
\put(5,65){\vector(1,0){24}}
\put(5,62){\vector(1,-1){24}} 
\put(40,65){\makebox(0,0){$\Lambda^{0,1}$}}
\put(40,35){\makebox(0,0){$\Lambda^{1,0}$}}
\put(40,50){\makebox(0,0){$\oplus$}} 
\put(55,65){\vector(1,0){24}}
\put(55,62){\vector(1,-1){24}} 
\put(55,35){\vector(1,0){24}}
\put(55,32){\vector(1,-1){24}} 
\put(65,23){\scriptsize$\partial$}
\put(65,53){\scriptsize$\partial$}
\put(65,67){\scriptsize$\overline\partial$} 
\put(65,37){\scriptsize$\overline\partial$}
\put(100,5){\makebox(0,0){$\begin{picture}(10,5)
\put(0,0){\line(1,0){10}} 
\put(0,5){\line(1,0){10}} 
\put(0,0){\line(0,1){5}}
\put(5,0){\line(0,1){5}} 
\put(10,0){\line(0,1){5}}
\end{picture}\,\Lambda^{1,0}$}}
\put(100,35){\makebox(0,0){$\Lambda_\perp^{1,1}$}}
\put(100,65){\makebox(0,0){$\begin{picture}(10,5)
\put(0,0){\line(1,0){10}} 
\put(0,5){\line(1,0){10}} 
\put(0,0){\line(0,1){5}}
\put(5,0){\line(0,1){5}} 
\put(10,0){\line(0,1){5}}
\end{picture}\,\Lambda^{0,1}$}}
\put(100,20){\makebox(0,0){$\oplus$}} 
\put(100,50){\makebox(0,0){$\oplus$}}
\put(120,60){\vector(1,-1){48}} 
\put(120,32){\vector(2,-1){48}} 
\put(120,5){\vector(1,0){48}}
\put(130,7){\scriptsize$\overline\partial^{(2)}$}
\put(130,28){\scriptsize$\partial\overline\partial$} 
\put(130,54){\scriptsize$\partial^{(2)}$} 
\put(215,5){\makebox(0,0){$\begin{picture}(10,5) 
\put(0,0){\line(1,0){10}}
\put(0,5){\line(1,0){10}} 
\put(0,0){\line(0,1){5}} 
\put(5,0){\line(0,1){5}}
\put(10,0){\line(0,1){5}}
\end{picture}\,\Lambda^{0,1}\otimes_\perp\begin{picture}(10,5)
\put(0,0){\line(1,0){10}} 
\put(0,5){\line(1,0){10}} 
\put(0,0){\line(0,1){5}}
\put(5,0){\line(0,1){5}} 
\put(10,0){\line(0,1){5}}
\end{picture}\,\Lambda^{1,0}$}}
\put(260,5){\vector(1,0){10}}
\put(277,5){\makebox(0,0){$0.$}}
\end{picture}}\end{equation}
The Penrose transform also computes the global cohomology of these complexes.
Since $H^0({\mathbb{F}}_{1,2}({\mathbb{C}}^3),\Theta)
={\mathfrak{sl}}(3,{\mathbb{C}})$ (arising from the infinitesimal action of
${\mathrm{SL}}(3,{\mathbb{C}})$ on ${\mathbb{F}}_{1,2}({\mathbb{C}})$ as a
complex homogeneous space) and all higher cohomology vanishes, we conclude that
$$0\to{\mathfrak{su}}(3)\to
\Gamma({\mathbb{CP}}_2,\Lambda^1)\xrightarrow{\;\nabla\;}
\Gamma({\mathbb{CP}}_2,\begin{picture}(10,5)
\put(0,0){\line(1,0){10}}
\put(0,5){\line(1,0){10}}
\put(0,0){\line(0,1){5}}
\put(5,0){\line(0,1){5}}
\put(10,0){\line(0,1){5}}
\end{picture}_{\,\circ}\Lambda^1)
\xrightarrow{\;\nabla^{(2)}\;}
\Gamma({\mathbb{CP}}_2,\begin{picture}(10,10)
\put(0,0){\line(1,0){10}}
\put(0,5){\line(1,0){10}}
\put(0,10){\line(1,0){10}}
\put(0,0){\line(0,1){10}}
\put(5,0){\line(0,1){10}}
\put(10,0){\line(0,1){10}}
\end{picture}_{\,\circ+}\Lambda^1)\to 0$$
is exact. It is the anti-self-dual deformation complex and we conclude that 
${\mathbb{CP}}_2$ is rigid as an anti-self-dual conformal manifold, also that 
the only conformal motions are infinitesimal isometries (this latter 
conclusion is well-known by other methods). 
\end{example}
\begin{remark}
The Penrose transform for ${\mathbb{CP}}_2$ may often be understood, as
typified in the previous example, in terms of its being an anti-self-dual
conformal manifold. Though we shall be able to construct the transform just as
well for $\tau:{\mathbb{F}}_{1,2}({\mathbb{C}}^{n+1})\to{\mathbb{CP}}_n$, it is
far from clear whether it can be interpreted via some intrinsic geometric
structure on~${\mathbb{CP}}_n$. The complexes of differential operators that 
arise are surely worthy of further study and interpretation.
\end{remark}
\begin{remark}
In order to keep the notation to a minimum, the rest of this article will be
confined to the case
$\tau:{\mathbb{F}}_{1,2}({\mathbb{C}}^4)\to{\mathbb{CP}}_3$. This already 
captures the difficulties in encountered for ${\mathbb{CP}}_n$ in general.
\end{remark}
\section{A spectral sequence}\label{spectral}
We shall end up following the Penrose transform for ${\mathbb{CP}}_2$
constructed in \cite{npb} but we begin with a general technique for which we
only need the following. Let $Z$ be a complex manifold and $M$ a smooth
manifold. Suppose
\begin{equation}\label{tau}\tau:Z\to M\end{equation}
is a smooth submersion with compact complex fibres. The mapping (\ref{twistor})
is the example we would like to understand. The classical example, however, is
the twistor fibration $\tau:{\mathbb{CP}}_3\to S^4$, viewed in \cite{mfa} as a
quaternionic version of the Hopf fibration.
\begin{theorem}
There is a smooth complex vector bundle $\Lambda_\mu^{1,0}$ on $Z$ that is
holomorphic along the fibres of~$\tau$. Let us denote by $\Lambda_\mu^{p,0}$ 
the $p$-fold exterior power of~$\Lambda_\mu^{1,0}$. Let us suppose that all 
$\tau_*^q\Lambda_\mu^{p,0}$ are smooth vector bundles on $M$ where $\tau_*^q$ 
denotes the $q^{\mathrm{th}}$ direct image with respect to the holomorphic 
structure along the fibres of~$\tau$. Then, there is a spectral sequence
\begin{equation}\label{ss}E_1^{p,q}=\Gamma(M,\tau_*^q\Lambda_\mu^{p,0})
\Longrightarrow H^{p+q}(Z,{\mathcal{O}})\end{equation}
whose differentials are linear differential operators on~$M$.
\end{theorem}
\begin{proof} To say that $\tau$ is a submersion is to say that
$d\tau:\tau^*\Lambda_M^1\to\Lambda_Z^1$ is injective. Let $\Lambda_\tau^1$
denote the cokernel of this homomorphism. It is the dual bundle to the bundle
of vertical vector fields. To say that the fibres of $\tau$ are holomorphic is 
to say that complex multiplication preserves the vertical vector fields. It
defines a surjection of vector bundles $\Lambda_Z^{0,1}\to\Lambda_\tau^{0,1}$ 
and thus a complex vector bundle as its kernel. The exact sequence
\begin{equation}\label{relativeforms}
0\to\Lambda_\mu^{1,0}\to\Lambda_Z^{0,1}\to\Lambda_\tau^{0,1}\to 0\end{equation}
defines the bundle $\Lambda_\mu^{1,0}$. The composition
$$\Lambda_Z^{0,0}\xrightarrow{\;\overline\partial\;}\Lambda_Z^{0,1}\to
\Lambda_\tau^{0,1}$$
coincides with the $\overline\partial$-operator intrinsic to the fibres
of~$\tau$, which we shall write as~$\overline\partial_\tau$. Commutativity of
the diagram 
$$\begin{array}{ccc}
\Lambda_Z^{0,1}&\xrightarrow{\;\overline\partial\;}&\Lambda_Z^{0,2}\\
\downarrow&&\downarrow\\
\Lambda_\tau^{0,1}&\xrightarrow{\;\overline\partial_\tau\;}&\Lambda_\tau^{0,2}
\end{array}$$
shows that $\overline\partial:\Lambda_Z^{0,1}\to\Lambda_Z^{0,2}$ induces a
differential operator $\overline\partial_\tau:\Lambda_\mu^{1,0}\to
\Lambda_\tau^{0,1}\otimes\Lambda_\mu^{1,0}$, which defines the holomorphic
structure along the fibres claimed in the statement of the theorem. The
spectral sequence is simply that of a filtered complex, where the assumption
that all $\tau_*^q\Lambda^{p,0}$ are smooth vector bundles is saying that the
dimensions of these finite-dimensional cohomology groups along the (compact)
fibres does not jump (and generically this is true).
\end{proof}

\begin{remark} We shall be interested (\ref{ss}) for the
fibration~(\ref{twistor}). In this case, because it is also true for the
fibration itself, all bundles are manifestly homogeneous under the action 
of ${\mathrm{SU}}(n+1)$. In particular, there can be no rank jumping and 
$\tau_*^q\Lambda_\mu^{p,0}$ are not only smooth vector bundles but homogeneous 
to boot. Their computation is a matter of representation theory, which we now 
pursue.
\end{remark}

Following~\cite{npb}, it is useful to consider the complexification of the 
submersion~(\ref{twistor}). For simplicity, we shall write out the details in 
case $n=3$ and abbreviate the flag manifold 
${\mathbb{F}}_{1,2}({\mathbb{C}}^4)$ as~${\mathbb{F}}$. A suitable 
complexification of ${\mathbb{CP}}_3$ is given by
$${\mathbb{M}}\equiv\{(\ell,H)\in{\mathbb{CP}}_3\times{\mathbb{CP}}_3^*
\mbox{ s.t.\ }\ell\not\in H\},$$
where ${\mathbb{CP}}_3^*$ is regarded as the space of hyperplanes $H$
in~${\mathbb{CP}}_3$ and the totally real embedding
${\mathbb{CP}}_3\hookrightarrow{\mathbb{M}}$ is given by
$\ell\mapsto(\ell,\ell^\perp)$. Just as ${\mathbb{CP}}_3$ is a homogeneous
space for ${\mathrm{SU}}(4)$, so ${\mathbb{M}}$ is a homogeneous space for 
${\mathrm{SL}}(4,{\mathbb{C}})$. Specifically, 
\begin{equation}\label{M}{\mathbb{M}}={\mathrm{SL}}(4,{\mathbb{C}})/
\mbox{\footnotesize$\left\{\left[{\begin{array}{cccc}*&0&0&0\\
0&*&*&*\\0&*&*&*\\0&*&*&*\end{array}}\right]\right\}$}
\quad\mbox{with basepoint}\quad
\mbox{\footnotesize$
\left(\left[\begin{array}c *\\ 0\\ 0\\ 0\end{array}\right],
\left[\begin{array}c 0\\ *\\ *\\ *\end{array}\right]\right)$}.\end{equation}
The submersion itself complexifies to a double fibration
\begin{equation}\label{db}\raisebox{-20pt}{\begin{picture}(80,50)
\put(0,5){\makebox(0,0){${\mathbb{F}}$}}
\put(80,5){\makebox(0,0){${\mathbb{M}}.$}}
\put(40,45){\makebox(0,0){${\mathbb{G}}$}}
\put(32,37){\vector(-1,-1){24}}
\put(48,37){\vector(1,-1){24}}
\put(16,28){\makebox(0,0){$\mu$}}
\put(66,28){\makebox(0,0){$\nu$}}
\end{picture}}\end{equation}
Here,
\begin{equation}\label{weirdchoice}{\mathbb{G}}={\mathrm{SL}}(4,{\mathbb{C}})/
\mbox{\footnotesize$\left\{\left[{\begin{array}{cccc}*&0&0&0\\
0&*&*&*\\0&0&*&*\\0&0&*&*\end{array}}\right]\right\}$}
\qquad{\mathbb{F}}={\mathrm{SL}}(4,{\mathbb{C}})/
\mbox{\footnotesize$\left\{\left[{\begin{array}{cccc}*&*&*&*\\
0&*&*&*\\0&0&*&*\\0&0&*&*\end{array}}\right]\right\}$},\end{equation}
the mapping $\nu:{\mathbb{G}}\to{\mathbb{M}}$ is the natural one, and
$\mu:{\mathbb{G}}\to{\mathbb{F}}$ is induced by the homomorphism
\begin{equation}\label{reflect}{\mathrm{SL}}(4,{\mathbb{C}})\ni A\mapsto
\mbox{\footnotesize$\left[{\begin{array}{cccc}0&1&0&0\\
1&0&0&0\\0&0&1&0\\0&0&0&1\end{array}}\right]$}
A\mbox{\footnotesize$\left[{\begin{array}{cccc}0&1&0&0\\
1&0&0&0\\0&0&1&0\\0&0&0&1\end{array}}\right]$}
\in{\mathrm{SL}}(4,{\mathbb{C}}).\end{equation}
Alternatively, a different choice of basepoint would lead to our writing
${\mathbb{F}}$ as 
\begin{equation}\label{easychoice}{\mathrm{SL}}(4,{\mathbb{C}})/
\mbox{\footnotesize$\left\{\left[{\begin{array}{cccc}*&0&*&*\\
*&*&*&*\\0&0&*&*\\0&0&*&*\end{array}}\right]\right\}$}\end{equation}
with both $\mu$ and $\nu$ the natural projections. But (\ref{weirdchoice}) is
preferred so as to coincide with the conventions established in
\cite{beastwood} for writing flag manifolds. Moreover, the irreducible
homogeneous vector bundles on ${\mathbb{F}}$ may then be written as
$$\xxo{a}{b}{c},\mbox{ for integers }a,b,c\mbox{ with }c\geq 0\enskip
\mbox{(see \cite{beastwood} for details)}.$$
The irreducible homogeneous vector bundles on ${\mathbb{G}}$ are in 1--1
correspondence with the finite-dimensional irreducible representations of the
stabiliser subgroup 
$$\mbox{\footnotesize$\left\{\left[{\begin{array}{cccc}*&0&0&0\\
0&*&*&*\\0&0&*&*\\0&0&*&*\end{array}}\right]\right\}$}$$
from~(\ref{weirdchoice}). But these are carried by the smaller subgroup
$$\mbox{\footnotesize$\left\{\left[{\begin{array}{cccc}*&0&0&0\\
0&*&0&0\\0&0&*&*\\0&0&*&*\end{array}}\right]\right\}$},$$
which, as the appropriate Levi factor, also carries the representations
inducing the irreducible homogeneous vector bundles on~${\mathbb{F}}$. We
shall, therefore, use the same notation for the homogeneous bundles on
${\mathbb{G}}$ but add ${\mathbb{G}}$ as a subscript when needed:--
$$\xxo{p}{q}{r}_{\mathbb{G}},\mbox{ for integers }p,q,r\mbox{ with }r\geq 0.$$
The price to pay for the conjugation (\ref{reflect}) used in defining $\mu$ is
that the corresponding simple reflection must be invoked in the Weyl group to
pull back homogeneous vector bundles from ${\mathbb{F}}$ to~${\mathbb{G}}$:--
\begin{equation}\label{pullback}
\mu^*\xxo{a}{b}{c}_{\mathbb{F}}=\xxo{-a}{a+b}{c}_{\mathbb{G}}\enskip
\mbox{(see \cite{beastwood} for details)}.\end{equation}
Finally, we need a notation for the homogeneous vector bundles
on~${\mathbb{M}}$ and its real slice~${\mathbb{CP}}_3$. {From} (\ref{M}) it is 
clear that we may use the usual notation~\cite{beastwood}
$$\xoo{a}{b}{c}_{\mathbb{M}},
\mbox{ for integers }a,b,c\mbox{ with }b\geq 0,c\geq 0$$
and that every ${\mathrm{SU}}(4)$-homogeneous bundle on the smooth manifold
${\mathbb{CP}}_3$ extends uniquely to a
${\mathrm{SL}}(4,{\mathbb{C}})$-homogeneous bundle on its
complexification~${\mathbb{M}}$. If we restrict the action of 
${\mathrm{SL}}(4,{\mathbb{C}})$ on the double fibration (\ref{db}) to the real 
form ${\mathrm{SU}}(4)$, then we obtain a real splitting $\sigma$ of the 
holomorphic submersion $\mu$ and a diagram
\begin{equation}\label{fullembedding}\begin{picture}(80,50)
\put(0,5){\makebox(0,0){${\mathbb{F}}$}}
\put(80,5){\makebox(0,0){${\mathbb{M}}.$}}
\put(40,45){\makebox(0,0){${\mathbb{G}}$}}
\put(32,37){\vector(-1,-1){24}}
\put(48,37){\vector(1,-1){24}}
\put(16,28){\makebox(0,0){$\mu$}}
\put(66,28){\makebox(0,0){$\nu$}}
\put(42,4){\makebox(0,0){${\mathbb{CP}}_3$}}
\put(20,6){\makebox(0,0){$\xrightarrow{\;\tau\;}$}}
\put(62,4){\makebox(0,0){$\hookrightarrow$}}
\put(10,11){\vector(1,1){24}}
\put(28,22){\makebox(0,0){$\scriptstyle\sigma$}}
\end{picture}\end{equation}
so that $\tau=\nu\circ\sigma$. In this way we may view ${\mathbb{F}}$ as a 
submanifold of ${\mathbb{G}}$ and then the fibres of~$\tau$, with their complex 
structure, coincide with the fibres of~$\nu$. The canonical homomorphism of vector 
bundles
$$\Lambda_{\mathbb{G}}^{1,0}|_{\mathbb{F}}=\sigma^*\Lambda_{\mathbb{G}}^{1,0}
\hookrightarrow
\sigma^*({\mathbb{C}}\otimes_{\mathbb{R}}\Lambda^1_{\mathbb{G}})
\xrightarrow{\;d\sigma\;}
{\mathbb{C}}\otimes_{\mathbb{R}}\Lambda^1_{\mathbb{F}}\twoheadrightarrow
\Lambda_{\mathbb{F}}^{0,1}$$
induces an isomorphism between $\Lambda_\mu^{1,0}$ on ${\mathbb{G}}$ defined by 
the exact sequence
\begin{equation}\label{extendedrelativeforms}
0\to\mu^*\Lambda_{\mathbb{F}}^{1,0}\xrightarrow{\;d\mu\;}
\Lambda_{\mathbb{G}}^{1,0}\to\Lambda_\mu^{1,0}\to 0\end{equation}
and its restriction to ${\mathbb{F}}$ defined by~(\ref{relativeforms}).
Furthermore, the holomorphic structure on $\Lambda_\mu^{1,0}$ along the fibres
of $\tau$ inherited from (\ref{relativeforms}) coincides with its evident
holomorphic structure on ${\mathbb{G}}$ defined
by~(\ref{extendedrelativeforms}) (notice from (\ref{fullembedding}) that the
fibres of $\tau$ agree with the fibres of $\nu$ over
${\mathbb{CP}}_3\hookrightarrow{\mathbb{M}}$). We reach the standard conclusion
that the spectral sequence (\ref{ss}) may be obtained by restricting the terms
in the spectral sequence
$$E_1^{p,q}=\Gamma({\mathbb{M}},\nu_*^q\Lambda_\mu^{p,0})$$
derived in \cite{beastwood} to~${\mathbb{CP}}_3$ replacing holomorphic sections
on the complexification ${\mathbb{M}}$ by smooth sections on the real
slice~${\mathbb{CP}}_3$. The point of this man{\oe}uvre is that the direct
image bundles $\nu_*^q\Lambda_\mu^p$ are easily computed by
representation theory as detailed in~\cite{beastwood}. Firstly, we
need to identify $\Lambda_\mu^{1,0}$ as a homogeneous vector bundle
on~${\mathbb{G}}$. According to (\ref{weirdchoice}) and (\ref{easychoice}) this
bundle is induced by the co-Adjoint representation of
$$\mbox{\footnotesize$\left\{\left[{\begin{array}{cccc}*&0&0&0\\
0&*&*&*\\0&0&*&*\\0&0&*&*\end{array}}\right]
\mbox{\normalsize$\in{\mathrm{SL}}(4,{\mathbb{C}})$}\right\}$}
\quad\mbox{on}\quad
\left[\frac{\mbox{\footnotesize$\left\{\left[{\begin{array}{cccc}*&0&*&*\\
*&*&*&*\\0&0&*&*\\0&0&*&*\end{array}}\right]
\mbox{\normalsize$\in{\mathfrak{sl}}(4,{\mathbb{C}})$}\right\}$}}
{\mbox{\footnotesize$\left\{\left[{\begin{array}{cccc}*&0&0&0\\
0&*&*&*\\0&0&*&*\\0&0&*&*\end{array}}\right]
\mbox{\normalsize$\in{\mathfrak{sl}}(4,{\mathbb{C}})$}\right\}$}}\right]^*$$
Therefore,
\begin{equation}\label{onG}\begin{array}{rcl}
\Lambda_\mu^{1,0}&=&
\xxo{1}{0}{1}_{\mathbb{G}}\oplus\xxo{-2}{1}{0}_{\mathbb{G}}\\[3pt]
\Lambda_\mu^{2,0}&=&
\xxo{2}{1}{0}_{\mathbb{G}}\oplus\xxo{-1}{1}{1}_{\mathbb{G}}\\[3pt]
\Lambda_\mu^{3,0}&=&
\xxo{0}{2}{0}_{\mathbb{G}}
\end{array}\end{equation}
and elementary application of the formul{\ae} in \cite{beastwood} yields
$$\begin{picture}(390,80)
\put(0,20){\makebox(0,0)[l]{$\Gamma({\mathbb{CP}}_3,\xoo{0}{0}{0})\to
\Gamma\Big({\mathbb{CP}}_3,\!\!
\begin{array}c\xoo{1}{0}{1}\\[-5pt] \oplus\\[-1pt] 
\xoo{-2}{1}{0}\end{array}\!\!\Big)\to
\Gamma\Big({\mathbb{CP}}_3,\!\!
\begin{array}c\xoo{2}{1}{0}\\[-5pt] \oplus\\[-1pt] 
\xoo{-1}{1}{1}\end{array}\!\!\Big)\to
\Gamma({\mathbb{CP}}_3,\xoo{0}{2}{0})$}}
\put(40,60){\makebox(0,0){$0$}}
\put(140,60){\makebox(0,0){$0$}}
\put(240,60){\makebox(0,0){$0$}}
\put(340,60){\makebox(0,0){$0$}}
\put(40,45){\makebox(0,0){$|$}}
\put(40,70){\vector(0,1){10}}
\put(380,20){\vector(1,0){10}}
\put(387,13){\makebox(0,0){$p$}}
\put(34,75){\makebox(0,0){$q$}}
\end{picture}$$
for the spectral sequence (\ref{ss}) applied to the submersion~(\ref{twistor}).
Convergence to
$$H^r({\mathbb{F}},{\mathcal{O}})=H^r({\mathbb{F}},\xxo{0}{0}{0})=
\Big\{\begin{array}l{\mathbb{C}}\enskip\mbox{if }r=0\\
0\enskip\mbox{if }r\geq 1,\end{array}$$
implies that this spectral sequence collapses to an exact sequence, which is
readily identified as~(\ref{lopsided}).

\section{Further examples}
As a natural higher dimensional counterpart to (\ref{asd_Hermitian}) let us now 
consider the Penrose transform of 
$H^r({\mathbb{F}}_{1,2}({\mathbb{C}}^4),\Theta)$. One immediate issue that must 
be dealt with is that the holomorphic tangent bundle $\Theta$ is reducible but 
not decomposable.
Specifically~\cite{beastwood}, there is a short exact sequence
\begin{equation}\label{Theta}
0\to\begin{array}c\xxo{2}{-1}{0}_{\mathbb{F}}\\[-5pt] \oplus\\[-1pt] 
\xxo{-1}{1}{1}_{\mathbb{F}}
\end{array}\to\Theta\to\xxo{1}{0}{1}_{\mathbb{F}}\to 0\end{equation}
that does not split as ${\mathrm{SL}}(4,{\mathbb{C}})$-homogeneous bundles. But
we can form the Penrose transform of each irreducible subfactor with the 
following results.
\begin{proposition}\label{one}
There is an exact sequence
$$0\to\Gamma({\mathbb{CP}}_3,\xoo{-2}{1}{0})\to
\Gamma\Big({\mathbb{CP}}_3,\!\!\!\!
\begin{array}c\xoo{-1}{1}{1}\\[-5pt] \oplus\\[-1pt] 
\xoo{-4}{2}{0}\end{array}\!\!\Big)\to
\Gamma\Big({\mathbb{CP}}_3,\!\!\!\!
\begin{array}c\xoo{0}{2}{0}\\[-5pt] \oplus\\[-1pt] 
\xoo{-3}{2}{1}\end{array}\!\!\Big)\to
\Gamma({\mathbb{CP}}_3,\xoo{-2}{3}{0})\to 0.$$
\end{proposition}
\begin{proof}
As usual~\cite{beastwood}, the spectral sequence (\ref{ss}) may be generalised
to incorporate a holomorphic vector bundle on ${\mathbb{F}}$ in the
coefficients and the discussion of \S\ref{spectral} is easily modified to
calculate the appropriate direct images. Specifically, the exact sequence we
are aiming for arises as the Penrose of the vector bundle
$\xxo{2}{-1}{0}_{\mathbb{F}}$. As a singular bundle, all its holomorphic
cohomology vanishes. To incorporate it into the Penrose transform, we firstly
use (\ref{pullback}) to conclude that
$$\mu^*\xxo{2}{-1}{0}_{\mathbb{F}}=\xxo{-2}{1}{0}_{\mathbb{G}}.$$
The modified spectral sequence is therefore
$$E_1^{p,q}=\Gamma({\mathbb{CP}}_3,\tau_*^q\Lambda_\mu^{p,0}(\xxo{-2}{1}{0}))
\Longrightarrow H^{p+q}({\mathbb{F}},\xxo{2}{-1}{0})=0.$$
{From} (\ref{onG}) we see that on~${\mathbb{G}}$
$$\Lambda_\mu^{\bullet,0}(\xxo{-2}{1}{0})=
\xxo{-2}{1}{0}_{\mathbb{G}}\to
\begin{array}c\xoo{-1}{1}{1}_{\mathbb{G}}\\[-5pt] \oplus\\[-1pt] 
\xoo{-4}{2}{0}_{\mathbb{G}}\end{array}\to
\begin{array}c\xoo{0}{2}{0}_{\mathbb{G}}\\[-5pt] \oplus\\[-1pt] 
\xoo{-3}{2}{1}_{\mathbb{G}}\end{array}\to
\xoo{-2}{3}{0}_{\mathbb{G}},$$
and the required conclusion follows by taking direct images as 
in~\cite{beastwood}.
\end{proof}
\begin{proposition}\label{two}
There is an exact sequence
$$0\to\Gamma({\mathbb{CP}}_3,\xoo{1}{0}{1})\to
\Gamma\Big({\mathbb{CP}}_3,\!\!\!\!
\begin{array}c\xoo{2}{0}{2}\\[-5pt] \oplus\\[-1pt] 
\xoo{2}{1}{0}\\[-5pt] \oplus\\[-1pt]\xoo{-1}{1}{1}\end{array}\!\!\Big)\to
\Gamma\Big({\mathbb{CP}}_3,\!\!
\begin{array}c\xoo{3}{1}{1}\\[-5pt] \oplus\\[-1pt] 
\xoo{0}{1}{2}\\[-5pt] \oplus\\[-1pt]\xoo{0}{2}{0}\end{array}\!\!\Big)\to
\Gamma({\mathbb{CP}}_3,\xoo{1}{2}{1})\to 0.$$
\end{proposition}
\begin{proof} Noting that 
$$\mu^*\xxo{-1}{1}{1}_{\mathbb{F}}=\xxo{1}{0}{1}_{\mathbb{G}},$$
the required conclusion is the Penrose transform of the singular
bundle~$\xxo{-1}{1}{1}_{\mathbb{F}}$. The extra bundles arise via various
tensor decompositions, such as
$$\Lambda_\mu^{1,0}\otimes\mu^*\xxo{-1}{1}{1}_{\mathbb{F}}\enskip=
\begin{array}c\xxo{1}{0}{1}_{\mathbb{G}}\\[-5pt] \oplus\\[-1pt]
\xxo{-2}{1}{0}_{\mathbb{G}}\end{array}\!\!\!\otimes\xxo{1}{0}{1}\enskip=
\begin{array}c\xxo{2}{0}{2}\\[-5pt] \oplus\\[-1pt]
\xxo{2}{1}{0}\\[-5pt] \oplus\\[-1pt]\xxo{-1}{1}{1}.\end{array}$$
Again, there are only $0^{\mathrm{th}}$ direct images in the spectral sequence.
\end{proof}
The remaining irreducible bundle from (\ref{Theta}) is 
$\xxo{1}{0}{1}_{\mathbb{F}}$ and its Penrose transform is as follows.
\begin{proposition}\label{three}
There is a exact sequence of  ${\mathrm{SL}}(4,{\mathbb{C}})$-modules
$$\begin{array}c\makebox[0pt][r]{$0\to\;$}
{\mathfrak{sl}}(4,{\mathbb{C}})\\
\downarrow\\
\Gamma({\mathbb{CP}}_3,\xoo{-1}{1}{1})\\ \mbox{ }\\ \mbox{ }\end{array}\to
\Gamma\Big({\mathbb{CP}}_3,\!\!\!\!
\begin{array}c\xoo{0}{1}{2}\\[-5pt] \oplus\\[-1pt] 
\xoo{0}{2}{0}\\[-5pt] \oplus\\[-1pt]\xoo{-3}{2}{1}\end{array}\!\!\Big)\to
\Gamma\Big({\mathbb{CP}}_3,\!\!\!\!
\begin{array}c\xoo{1}{2}{1}\\[-5pt] \oplus\\[-1pt] 
\xoo{-2}{2}{2}\\[-5pt] \oplus\\[-1pt]\xoo{-2}{3}{0}\end{array}\!\!\Big)\to
\Gamma({\mathbb{CP}}_3,\xoo{-1}{3}{1})\to 0.$$
\end{proposition}
\begin{proof}Apply the standard machinery~\cite{beastwood} to compute
$$E_1^{p,q}=
\Gamma({\mathbb{CP}}_3,
\tau_*^q\Lambda_\mu^{p,0}(\mu^*\xxo{1}{0}{1}_{\mathbb{F}}))=
\Gamma({\mathbb{CP}}_3,
\tau_*^q\Lambda_\mu^{p,0}(\xxo{-1}{1}{1}))
\Longrightarrow H^{p+q}({\mathbb{F}},\xxo{1}{0}{1}_{\mathbb{F}}).$$
By the Bott-Borel-Weil Theorem, in the conventions of~\cite{beastwood}, 
$$H^r({\mathbb{F}},\xxo{1}{0}{1}_{\mathbb{F}})=\Big\{
\begin{array}l\ooo{1}{0}{1}
={\mathfrak{sl}}(4,{\mathbb{C}})\enskip\mbox{if }r=0\\
0\enskip\mbox{if }r\geq 1
\end{array}$$
and the required exact sequence emerges.\end{proof}
Propositions \ref{one}, \ref{two}, and~\ref{three} may be combined to give the 
following result for the Penrose transform of 
$H^r({\mathbb{F}}_{1,2}({\mathbb{C}}^{4}))$ in terms of differential operators 
on~${\mathbb{CP}}_3$.   
\begin{theorem}\label{uglytheorem}
There is a complex of differential operators
on~${\mathbb{CP}}_3$ 
$$\raisebox{-30pt}{\begin{picture}(417,102)(10,0)
\put(10,65){\makebox(0,0){$0$}}
\put(15,65){\vector(1,0){14}}
\put(15,62){\vector(2,-3){14}} 
\put(40,65){\makebox(0,0){$\Lambda^{0,1}$}}
\put(40,35){\makebox(0,0){$\Lambda^{1,0}$}}
\put(40,50){\makebox(0,0){$\oplus$}} 
\put(55,68){\vector(1,1){24}}
\put(55,65){\vector(1,0){24}}
\put(55,62){\vector(1,-1){24}} 
\put(55,35){\vector(1,0){24}}
\put(55,32){\vector(1,-1){24}}
\put(65,23){\scriptsize$\partial$}
\put(65,53){\scriptsize$\partial$}
\put(65,67){\scriptsize$\overline\partial$} 
\put(65,85){\scriptsize$\overline\partial$} 
\put(65,37){\scriptsize$\overline\partial$}
\put(110,5){\makebox(0,0){$\begin{picture}(10,5)
\put(0,0){\line(1,0){10}} 
\put(0,5){\line(1,0){10}} 
\put(0,0){\line(0,1){5}}
\put(5,0){\line(0,1){5}} 
\put(10,0){\line(0,1){5}}
\end{picture}\,\Lambda^{1,0}$}}
\put(110,35){\makebox(0,0){$\Lambda^{0,1}\!\otimes_\perp\!\Lambda^{1,0}$}}
\put(110,65){\makebox(0,0){$\begin{picture}(10,5)
\put(0,0){\line(1,0){10}} 
\put(0,5){\line(1,0){10}} 
\put(0,0){\line(0,1){5}}
\put(5,0){\line(0,1){5}} 
\put(10,0){\line(0,1){5}}
\end{picture}\,\Lambda^{0,1}$}}
\put(110,95){\makebox(0,0){$\begin{picture}(5,10)
\put(0,0){\line(1,0){5}} 
\put(0,5){\line(1,0){5}} 
\put(0,10){\line(1,0){5}}
\put(0,0){\line(0,1){10}} 
\put(5,0){\line(0,1){10}}
\end{picture}\,\Lambda^{0,1}$}}
\put(110,20){\makebox(0,0){$\oplus$}} 
\put(110,50){\makebox(0,0){$\oplus$}}
\put(110,80){\makebox(0,0){$\oplus$}}
\put(140,95){\vector(1,0){24}}
\put(140,63){\vector(1,1){24}}
\put(140,60){\vector(3,-2){70}} 
\put(140,32){\vector(3,-1){70}} 
\put(140,5){\vector(1,0){70}}
\put(140,35){\vector(3,2){32}}
\put(140,88){\vector(3,-2){32}}
\put(150,97){\scriptsize$\overline\partial$}
\put(143,72){\scriptsize$\overline\partial$}
\put(144,42){\scriptsize$\overline\partial$}
\put(163,74){\scriptsize$\partial$}
\put(150,7){\scriptsize$\overline\partial^{(2)}$}
\put(165,24){\scriptsize$\partial\overline\partial$} 
\put(180,35){\scriptsize$\partial^{(2)}$} 
\put(190,95){\makebox(0,0){$\begin{picture}(10,10)
\put(0,0){\line(1,0){5}} 
\put(0,5){\line(1,0){10}} 
\put(0,10){\line(1,0){10}}
\put(0,0){\line(0,1){10}} 
\put(5,0){\line(0,1){10}}
\put(10,5){\line(0,1){5}}
\end{picture}\,\Lambda^{0,1}$}}
\put(195,80){\makebox(0,0){$\oplus$}}
\put(210,65){\makebox(0,0){$\begin{picture}(5,10)
\put(0,0){\line(1,0){5}} 
\put(0,5){\line(1,0){5}} 
\put(0,10){\line(1,0){5}}
\put(0,0){\line(0,1){10}} 
\put(5,0){\line(0,1){10}}
\end{picture}\,\Lambda^{0,1}\!\otimes_\perp\!\Lambda^{1,0}$}}
\put(230,35){\makebox(0,0){$\oplus$}}
\put(255,5){\makebox(0,0){$\begin{picture}(10,5) 
\put(0,0){\line(1,0){10}}
\put(0,5){\line(1,0){10}} 
\put(0,0){\line(0,1){5}} 
\put(5,0){\line(0,1){5}}
\put(10,0){\line(0,1){5}}
\end{picture}\,\Lambda^{0,1}\otimes_\perp\begin{picture}(10,5)
\put(0,0){\line(1,0){10}} 
\put(0,5){\line(1,0){10}} 
\put(0,0){\line(0,1){5}}
\put(5,0){\line(0,1){5}} 
\put(10,0){\line(0,1){5}}
\end{picture}\,\Lambda^{1,0}$}}
\put(300,5){\vector(1,0){20}}
\put(309,8){\scriptsize$\overline\partial$}
\put(365,5){\makebox(0,0){$\begin{picture}(10,10) 
\put(0,5){\line(1,0){10}}
\put(0,10){\line(1,0){10}} 
\put(0,0){\line(0,1){10}} 
\put(5,0){\line(0,1){10}}
\put(10,5){\line(0,1){5}}
\put(0,0){\line(1,0){5}}
\end{picture}\,\Lambda^{0,1}\otimes_\perp\begin{picture}(10,5)
\put(0,0){\line(1,0){10}} 
\put(0,5){\line(1,0){10}} 
\put(0,0){\line(0,1){5}}
\put(5,0){\line(0,1){5}} 
\put(10,0){\line(0,1){5}}
\end{picture}\,\Lambda^{1,0}$}}
\put(410,5){\vector(1,0){10}}
\put(427,5){\makebox(0,0){$0$}}
\end{picture}}$$
whose {global\/} $0^{\mathrm{th}}$ cohomology is
${\mathfrak{sl}}(4,{\mathbb{C}})$ and is otherwise exact. The operators in this
complex are the natural ones induced by the Fubini-Study
connection.\end{theorem}
\begin{proof}
To compute the Penrose transform of $H^r({\mathbb{F}},\Theta)$ we need to 
identify the complex $\Lambda^{\bullet,0}(\mu^*\Theta)$ and compute direct 
images $\tau_*^q\Lambda^{p,0}(\mu^*\Theta)$ down on~${\mathbb{CP}}_3$. {From} 
(\ref{Theta}) and (\ref{onG}) we obtain
$$\begin{array}{ccccccc}
\Lambda^{0,0}(\mu^*\Theta)&\xrightarrow{\;\partial_\mu\;}&
\Lambda^{1,0}(\mu^*\Theta)&\xrightarrow{\;\partial_\mu\;}&
\Lambda^{2,0}(\mu^*\Theta)&\xrightarrow{\;\partial_\mu\;}&
\Lambda^{3,0}(\mu^*\Theta)\\
\|&&\|&&\|&&\|\\
&&\xxo{0}{1}{2}\makebox[0pt][l]{$\;*$}
&&\xxo{1}{2}{1}\makebox[0pt][l]{$\;*$}\\[-5pt]
&&\oplus&&\oplus\\[-1pt]
\xxo{-1}{1}{1}\makebox[0pt][l]{$\;*$}&&\xxo{0}{2}{0}\makebox[0pt][l]{$\;*$}
&&\xxo{-2}{2}{2}&&\xxo{-1}{3}{1}\\[-5pt]
&&\oplus&&\oplus\\[-1pt]
&&\xxo{-3}{2}{1}\makebox[0pt][l]{$\;*$}
&&\xxo{-2}{3}{0}\makebox[0pt][l]{$\;*$}\\[-5pt]
\oplus&&\oplus&&\oplus&&\oplus\\[-1pt]
&&\xxo{2}{0}{2}&&\xxo{3}{1}{1}\\[-5pt]
&&\oplus&&\oplus\\[-1pt]
\xxo{1}{0}{1}&&\xxo{2}{1}{0}&&\xxo{0}{1}{2}&&\xxo{1}{2}{1}\\[-5pt]
&&\oplus&&\oplus\\[-1pt]
\oplus&&\xxo{-1}{1}{1}&&\xxo{0}{2}{0}&&\oplus\\[-5pt]
&&\oplus&&\oplus\\[-1pt]
\xxo{-2}{1}{0}&&\xxo{-1}{1}{1}&&\xxo{0}{2}{0}&&\xxo{-2}{3}{0}.\\[-5pt]
&&\oplus&&\oplus\\[-1pt]
&&\xxo{-4}{2}{0}&&\xxo{-3}{2}{1}
\end{array}$$
Notice that all the bundles marked $*$ are repeated in the next column. A more
detailed analysis shows that, when restricted to these bundles and an
appropriate target bundle, the differential operator $\partial_\mu$ is simply
the identity mapping. Also, notice that all the bundles have only
$0^{\mathrm{th}}$ direct images down on~${\mathbb{CP}}_3$. Diagram chasing,
either on ${\mathbb{G}}$ or down on~${\mathbb{CP}}_3$, now shows that all the
bundles labelled $*$ can be eliminated from the resulting complex at the
expense of introducing second order operators. (This construction is similar to
that of the Bernstein-Gelfand-Gelfand operators in \cite{beastwood} but there
appears to be no precise link.) The complex in the statement of the theorem is
obtained by using the more traditional notation for irreducible Hermitian
bundles (the differential operators in this complex are the only possible
Hermitian-invariant ones). Its cohomology realises $H^r({\mathbb{F}},\Theta)$.
But from (\ref{Theta}) and the algorithms of \cite{beastwood} it
follows that $H^0({\mathbb{F}},\Theta)=\ooo{1}{0}{1}$ and all higher cohomology
vanishes.
\end{proof}
The complex in Theorem~\ref{uglytheorem} is somewhat ugly compared to the
pleasing complex (\ref{asd_Hermitian}) on~${\mathbb{CP}}_2$. There does appear
to be a more satisfying analogue as follows.
$$\raisebox{-30pt}{\begin{picture}(417,70)(10,0)
\put(10,65){\makebox(0,0){$0$}}
\put(15,65){\vector(1,0){14}}
\put(15,62){\vector(2,-3){14}} 
\put(40,65){\makebox(0,0){$\Lambda^{0,1}$}}
\put(40,35){\makebox(0,0){$\Lambda^{1,0}$}}
\put(40,50){\makebox(0,0){$\oplus$}} 
\put(55,65){\vector(1,0){24}}
\put(55,62){\vector(1,-1){24}} 
\put(55,35){\vector(1,0){24}}
\put(55,32){\vector(1,-1){24}}
\put(110,5){\makebox(0,0){$\begin{picture}(10,5)
\put(0,0){\line(1,0){10}} 
\put(0,5){\line(1,0){10}} 
\put(0,0){\line(0,1){5}}
\put(5,0){\line(0,1){5}} 
\put(10,0){\line(0,1){5}}
\end{picture}\,\Lambda^{1,0}$}}
\put(110,35){\makebox(0,0){$\Lambda^{0,1}\!\otimes_\perp\!\Lambda^{1,0}$}}
\put(110,65){\makebox(0,0){$\begin{picture}(10,5)
\put(0,0){\line(1,0){10}} 
\put(0,5){\line(1,0){10}} 
\put(0,0){\line(0,1){5}}
\put(5,0){\line(0,1){5}} 
\put(10,0){\line(0,1){5}}
\end{picture}\,\Lambda^{0,1}$}}
\put(110,20){\makebox(0,0){$\oplus$}} 
\put(110,50){\makebox(0,0){$\oplus$}}
\put(140,60){\vector(3,-2){70}} 
\put(140,64){\vector(1,0){70}} 
\put(140,62){\vector(3,-1){70}} 
\put(140,32){\vector(3,-1){70}} 
\put(140,34){\vector(1,0){70}} 
\put(140,5){\vector(1,0){70}}
\put(255,20){\makebox(0,0){$\oplus$}}
\put(255,50){\makebox(0,0){$\oplus$}}
\put(255,5){\makebox(0,0){$\begin{picture}(10,5) 
\put(0,0){\line(1,0){10}}
\put(0,5){\line(1,0){10}} 
\put(0,0){\line(0,1){5}} 
\put(5,0){\line(0,1){5}}
\put(10,0){\line(0,1){5}}
\end{picture}\,\Lambda^{0,1}\otimes_\perp\begin{picture}(10,5)
\put(0,0){\line(1,0){10}} 
\put(0,5){\line(1,0){10}} 
\put(0,0){\line(0,1){5}}
\put(5,0){\line(0,1){5}} 
\put(10,0){\line(0,1){5}}
\end{picture}\,\Lambda^{1,0}$}}
\put(255,35){\makebox(0,0){$\begin{picture}(10,10) 
\put(0,5){\line(1,0){10}}
\put(0,10){\line(1,0){10}} 
\put(0,0){\line(0,1){10}} 
\put(5,0){\line(0,1){10}}
\put(10,5){\line(0,1){5}}
\put(0,0){\line(1,0){5}}
\end{picture}\,\Lambda^{0,1}\otimes_\perp
\Lambda^{1,0}$}}
\put(255,65){\makebox(0,0){$\begin{picture}(10,10) 
\put(0,0){\line(1,0){10}}
\put(0,5){\line(1,0){10}} 
\put(0,0){\line(0,1){10}} 
\put(5,0){\line(0,1){10}}
\put(10,0){\line(0,1){10}}
\put(0,10){\line(1,0){10}}
\end{picture}\,\Lambda^{0,1}$}}
\put(300,5){\vector(1,0){20}}
\put(300,34){\vector(1,0){20}}
\put(300,30){\vector(1,-1){20}}
\put(290,60){\vector(4,-3){30}}
\put(365,35){\makebox(0,0){$\begin{picture}(10,10) 
\put(0,0){\line(1,0){10}}
\put(0,5){\line(1,0){10}} 
\put(0,0){\line(0,1){10}} 
\put(5,0){\line(0,1){10}}
\put(10,0){\line(0,1){10}}
\put(0,10){\line(1,0){10}}
\end{picture}\,\Lambda^{0,1}\otimes_\perp\Lambda^{1,0}$}}
\put(365,20){\makebox(0,0){$\oplus$}}
\put(365,5){\makebox(0,0){$\begin{picture}(10,10) 
\put(0,5){\line(1,0){10}}
\put(0,10){\line(1,0){10}} 
\put(0,0){\line(0,1){10}} 
\put(5,0){\line(0,1){10}}
\put(10,5){\line(0,1){5}}
\put(0,0){\line(1,0){5}}
\end{picture}\,\Lambda^{0,1}\!\otimes_\perp\!\begin{picture}(10,5)
\put(0,0){\line(1,0){10}} 
\put(0,5){\line(1,0){10}} 
\put(0,0){\line(0,1){5}}
\put(5,0){\line(0,1){5}} 
\put(10,0){\line(0,1){5}}
\end{picture}\,\Lambda^{1,0}$}}
\put(410,5){\vector(1,0){10}}
\put(405,25){\vector(1,-1){15}}
\put(427,5){\makebox(0,0){$0$.}}
\end{picture}}$$
Though the details are not yet worked out, it appears that this complex may
also be obtained from the Penrose transform of a suitable homogeneous bundle
$V$ on~${\mathbb{F}}$. This bundle $V$ arises as extension
$$0\to\begin{array}c\xxo{2}{-1}{0}_{\mathbb{F}}\\[-5pt] \oplus\\[-1pt] 
\xxo{-1}{1}{1}_{\mathbb{F}}
\end{array}\to V\to
\begin{array}c\xxo{-2}{3}{0}_{\mathbb{F}}\\[-5pt] \oplus\\[-1pt] 
\xxo{1}{0}{1}_{\mathbb{F}}
\end{array}\to 0.$$
If true, this would gives rise to some $1^{\mathrm{st}}$ cohomology too since
$$H^1({\mathbb{F}},V)=H^1({\mathbb{F}},\xxo{-2}{3}{0})=\ooo{0}{2}{0}.$$
For the moment, this remains a conjecture. The geometric significance of the
vector bundle~$V$, if any, also remains unclear. Presumably, this more 
satisfactory complex on ${\mathbb{CP}}_3$ is elliptic.

We conclude with some final geometric observations, which also provided the
main motivation for this study. The complex (\ref{asd_Hermitian}) has a simple
real form (\ref{asd_deformation}). Whilst this is not true for the complex just
suggested on~${\mathbb{CP}}_3$ as a satisfactory analogue, it is still true
that the first operator is simply (\ref{confKilling}) but acting on
complex-valued $1$-forms. This is the conformal Killing operator. The Killing
operator is similar but does not remove the trace:--
$$\Lambda^1\to
\begin{picture}(10,5)
\put(0,0){\line(1,0){10}}
\put(0,5){\line(1,0){10}}
\put(0,0){\line(0,1){5}}
\put(5,0){\line(0,1){5}}
\put(10,0){\line(0,1){5}}
\end{picture}\,\Lambda^1\quad\mbox{by}\quad
\phi_a\longmapsto\nabla_a\phi_b+\nabla_b\phi_a.$$
In~\cite{kokyuroku} these observations were used establish results concerning
real integral geometry on~${\mathbb{CP}}_2$. Unfortunately, as we have
seen, the Penrose transform for ${\mathbb{CP}}_n$ when $n\geq 3$
yields much more complicated results. Fortunately, for the purposes of real 
integral geometry on~${\mathbb{CP}}_n$, other methods found in 
joint work with Hubert Goldschmidt \cite{eg} have circumvented this approach.

\end{document}